\documentclass{amsart}
\usepackage{amscd,amssymb,graphicx, amsmath}
\author[Fialowski]{Alice Fialowski}
\address{
Alice Fialowski\\
E\"otv\"os Lor\'and University
Budapest and University P\'ecs, Hungary}
\email{fialowsk@cs.elte.hu, fialowsk@ttk.pte.hu}
\author[Mandal]{Ashis Mandal}
\address{
Ashis Mandal\\
Indian Institute of Technology
Kanpur, India}
\email{amandal@iitk.ac.in, ashismandal.r@gmail.com}

\newtheorem{pf}{Proof}
\newtheorem{thm}{Theorem}[section]

\theoremstyle{definition}
\newtheorem{prop}[thm]{Proposition}
\newtheorem{dfn}[thm]{Definition}

\newtheorem{rem}{Remark}

\def \C{\mbox{$\mathbb C$}}

\def \N{\mbox{$\mathbb N$}}
\def \sl{\mathfrak{sl}}

\def\and{\mbox{ \rm and }}

\def\g{\mathcal g}

\def\g{\mathfrak{g}}

\input{xy}
\xyoption{all}

\begin{document}
\setlength{\multlinegap}{0pt}
\title[On metric Leibniz algebras and deformations]
{On Metric Leibniz Algebras and Deformations}

\address{}%
\email{}%

\thanks{}%

\date{\today}
\begin{abstract}
In this note we consider low dimensional metric Leibniz algebras with an invariant inner product over the complex numbers up to dimension $5$. We study their deformations, and give explicit formulas for the cocycles and deformations. We identify among those the metric deformations.
\end{abstract}
\maketitle

\section{Introduction}
Leibniz algebras with an invariant inner product form a very important special class of Leibniz algebras. They have the extra property that they have an invariant, symmetric, non-degenerate bilinear form. We know that any semisimple Leibniz algebra has such a form, but there are also other objects with this property. In this work, we consider Leibniz algebras over the complex numbers. In that case, all inner products are equivalent.
Lie or Leibniz algebras with an invariant inner product are also referred to as \emph{metric} Lie/Leibniz algebras or \emph{quadratic}, \emph{symmetric self-dual}, \emph{metric kinematical} Lie algebras in the literature (see \cite{MR, FS, Bor, FO}), and include the diamond and oscillator algebras as examples of solvable algebras with an invariant inner product. For a good survey on metric Lie algebras and other metric non-associative structures see \cite{O1, Bor}. On the moduli space of low dimensional metric Lie algebras see \cite{FP}.

These Lie or Leibniz algebras are interesting for several reasons. Let $G$ be the corresponding connected and simply connected Lie group of the Lie algebra $\g$. The inner product on $\g$ induces a left-invariant Riemannian metric on $G$ in a natural way. Such Lie algebras correspond to special pseudo-Riemannian symmetric spaces, see \cite{KO3, O2}.
They also show up in the formulation of some physical problems, like in the Adler-Kostant-Symes scheme \cite{A, K, S, O2} or in gauge theory and conformal field theory, as precisely the Lie algebras for which a Sugawara construction exists \cite{Su, FO}. An analogous setup also appears for Leibniz algebras.

Metric Leibniz algebras are also of great importance.  A good summery of the known results are in \cite{BH1, BH2}. In applications in physics, low dimensional metric Leibniz algebras are especially important. In describing such objects, deformation theory is a good tool. With the help of deformations, we can find new objects, and among those search for the metric ones. The infinitesimal deformations are characterized by special $2$-cocycles, which form a metric object with invariant, non-degenerate, symmetric bilinear form. If the deformation has no higher order terms ( terms with $t^n$ coefficients and $n\geq 2$), then the infinitesimal deformation gives  an honest deformation. 

In our paper we compute metric Leibniz deformations of $2$, $3$, $4$ and $5$-dimensional 
metric Leibniz algebras, including the Lie ones, as they also can give Leibniz deformations (see \cite{FM}). A list of nilpotent indecomposable $5$-dimensional metric non-Lie Leibniz algebras exists in \cite{D}. There are altogether 267 non-isomorphic algebras, out of those $149$ are single ones and $118$ are families.  Among them, there are no metric ones. With our deformation method we did find two new indecomposable nilpotent metric Leibniz (non-Lie) algebras which were not listed in \cite{D}. A solvable and non-Lie algebra example of Leibniz algebras in dimension $5$ must have $3$- or $4$-dimensional nilradical. They are partially listed in several papers, see \cite {CLOK1, CLOK2, KhSh, KhRH}. Altogether they give the entire classification. Those lists of solvable metric Leibniz algebras, - except the  decomposable ones - do not give new metric Leibniz algebras. 

The structure of the paper is as follows. In the first section we briefly recall the basic definitions and known results on quadratic Leibniz algebras and deformations. In Section $2$, we describe all nonequivalent metric Leibniz deformations of $2$- and $3$-dimensional Leibniz algebras. Section $3$ deals with the $4$-dimensional cases. Here we have a complete list of metric Leibniz  deformations of Leibniz (and Lie) algebras. The last Section gives metric Leibniz algebras in dimension $5$, and describes all their metric deformations.

\bigskip

{\bf Acknowledgement:} This work was initiated  during a visit of Ashis Mandal to  E\"otv\"os Lor\'and University, Hungary and has been finalized at the Max Planck Institute at Bonn, Germany. The authors acknowledge the generous support and  very active working environment  of these two institutes.  The second author  would like to thank for the support from MATRICS -Research grant  MTR/2018/000510 by SERB, DST, Government of India.
\section{Preliminaries}
\subsection{Leibniz algebras}
\dfn
Let $\mathfrak{L}$ be a vector space with a bilinear product denoted by the bracket  $[-,-]: \mathfrak{L} \times \mathfrak{L} \to \mathfrak{L}$, which satisfies ( the left Leibniz identity )
$$
[x,[y,z]]=[[x,y],z] + [y,[x,z]]
$$
for any $x,y,z \in \mathfrak{L}$. We call $(\mathfrak{L}, [-,-])$  a \emph{left Leibniz algebra}. If one has
$$
[x,[y,z]]=[[x,y],z]-[[x,z],y]
$$
for every $x,y,z \in \mathfrak{L}$, we call it a \emph{right Leibniz algebra}. If both equations are satisfied, we call the Leibniz algebra $\mathfrak{L}$ is \emph{symmetric}.

\rem
Obviously Lie algebras are special case of symmetric Leibniz algebras with the bilinear product being additionally skew-symmetric .

\dfn
For a given Leibniz algebra $\mathfrak{L}$ define
$$
\mathfrak{L}^1=\mathfrak{L}, ~~\mbox{and}~\quad \mathfrak{L}^{k+1}=[\mathfrak{L}^k, \mathfrak{L}]~~\mbox{for}~~ k \geq 1.
$$
We call $\mathfrak{L}$ is \emph{nilpotent} if there exists $n \in \N$ such that $\mathfrak{L}^n=0$.
\dfn
Define 
$$
\mathfrak{L}^{[1]}=\mathfrak{L}, ~~\mbox{and}~ \quad \mathfrak{L}^{[k+1]}=[\mathfrak{L}^{[k]}, \mathfrak{L}^{[k]}] ~~\mbox{for}~~ k \geq 1.
$$
We call $\mathfrak{L}$ is \emph{solvable} if there exists $m \in \N$ such that $\mathfrak{L}^m=0$.
\rem
The \emph{nilradical} of a Leibniz algebra is the maximal nilpotent ideal.
\rem
Nilpotent Leibniz algebras are obviously solvable.
\dfn
Let $(\mathfrak{L}, [-,-])$ be a Leibniz algebra (left or right). Denote the vector space spanned by the set $\{[x,x] \slash x \in \mathfrak{L}\}$ by $\mathfrak{I}_{\mathfrak{L}}$ and call it the \emph{Leibniz kernel} of $\mathfrak{L}$.

We say that a Leibniz algebra is \emph{simple} if $[\mathfrak{L}, \mathfrak{L}] \neq \mathfrak{I}_{\mathfrak{L}}$, and $\mathfrak{L}$ has no ideals beside ${0}, \mathfrak{L}$ and $\mathfrak{I}_{\mathfrak{L}}$.
We call $\mathfrak{L}$ \emph{semisimple}, if the maximal solvable ideal of $\mathfrak{L}$ is $\mathfrak{I}_{\mathfrak{L}}$.

\rem
By Levi's Theorem, every Leibniz algebra is a direct sum of its solvable radical (maximal solvable ideal) and a semisimple Leibniz algebra, see \cite{Ba}.

\dfn
Let $(\mathfrak{L},[-,-])$ be a left or right Leibniz algebra and  $B:\mathfrak{L} \times \mathfrak{L} \to \C$ a bilinear form on $\mathfrak{L}$. If
$$
B([x,y],z)=B(x,[y,z])
$$
for every $x,y,z \in \mathfrak{L}$, we call such a bilinear form an \emph{ invariant} form. If $\mathfrak{L}$ has a non-degenerate  and invariant bilinear form, we call $(\mathfrak{L}, B)$ a \emph{metric Leibniz algebra}.

\rem
Of course, for a metric Leibniz algebra there can be many choices of invariant forms.

\begin{prop} (see \cite{BH1})
Let $\mathfrak{L}$ be a Leibniz algebra. 
\begin{itemize}
\item[(i)] A metric Leibniz algebra $(\mathfrak{L},B)$ is necesserily symmetric. 

\item[(ii)] The Leibniz algebra $\mathfrak{L}$ is (semi)simple if and only if $\mathfrak{I}_{\mathfrak{L}}$ is a (semi)simple Lie algebra.

\end{itemize}
\end{prop}
\subsection {Leibniz cohomology and deformations}
Let  $\mathfrak{L}$ be a Leibniz algebra (right Leibniz algebra) and $M$ be a representation of $\mathfrak{L}$. By definition (\cite {LP}), $M$ is a $\mathbb{C}$-module equipped with two actions (left and right) of $\mathfrak{L}$, 
$$[-,-]: \mathfrak{L} \times M\longrightarrow M~~\mbox{and}~[-,-]:M \times \mathfrak{L} \longrightarrow M ~~\mbox{such that the expression}~$$ 
$$[x,[y,z]]=[[x,y],z]-[[x,z],y]$$
holds, whenever one of the variables is from $M$ and the others from $\mathfrak{L}$.

 In particular, $\mathfrak{L}$ is a representation of itself with the obvious action given by the Leibniz bracket in $\mathfrak{L}$, which is the adjoint representation of $\mathfrak{L}$.

\dfn
\cite{FMM} The space of \emph{Leibniz $n$-cochains} $CL^{*}(\mathfrak{L}, M)$ of a Leibniz algebra $\mathfrak{L}$ with coefficients in an $\mathfrak{L}$- module $M$ is given by 
 $CL^n({\mathfrak{L}}; {M}):= \mbox{Hom} _\mathbb{C}({\mathfrak{L}}^{\otimes n}, {M})$, $n\geq 0.$ Let 
$\delta^n : CL^n({\mathfrak{L}}; {M})\longrightarrow CL^{n+1}(\mathfrak{L}; M)$
be a $\mathbb{C}$-homomorphism defined by 
\begin{equation*}
\begin{split}
&\delta^nf(x_1,\cdots,x_{n+1})\\
&:= [x_1,f(x_2,\cdots,x_{n+1})] + \sum_{i=2}^{n+1}(-1)^i[f(x_1,\cdots,\hat{x}_i,\cdots,x_{n+1}),x_i]\\
&+ \sum_{1\leq i<j\leq n+1}(-1)^{j+1}f(x_1,\cdots,x_{i-1},[x_i,x_j],x_{i+1},\cdots,\hat{x}_j,\cdots, x_{n+1}).
\end{split}
\end{equation*}
Then $(CL^*(\mathfrak{L}; M),\delta)$ is a cochain complex, whose \emph{cohomology} is called the cohomology of the Leibniz algebra $\mathfrak{L}$ with coefficients in the module $M$. The $n$-th cohomology is denoted by $HL^n(\mathfrak{L}; M)$. In particular, the  $n$-th cohomology of $\mathfrak{L}$ with coefficients in itself is denoted by $HL^n(\mathfrak{L}; \mathfrak{L}).$

For studying deformations, we consider this special case $M=\mathfrak{L}$ with the module structure given by bracket $[-,-]$ of $\mathfrak{L}$.
\begin{dfn}
 A \emph{formal $1$-parameter family of deformations} of a Leibniz algebra $\mathfrak{L}$ (see \cite{FMM, FM}) is a Leibniz bracket $\mu_t$ on the $\mathbb{C}[[t]]$-module $\mathfrak{L}_t= \mathfrak{L}{\otimes}_{\mathbb C} \mathbb{C}[[t]]$ such that it is a formal power series expressed as 
 $$
\mu_t( -,- )=\mu_0(-, -) + t\phi_1(-,-) + t^2\phi_2(-,-) + \cdots ,
$$
where $\mu_0$ is the original Leibniz bracket, and all $\phi_i \in CL^2(\mathfrak{L}; \mathfrak{L})$ are $2$-cochains. 
  \end{dfn}
\begin{dfn}
A $1$-parameter deformation is called a \emph{jump deformation}, if for any  nonzero values $s$ and $u$ of the parameter $t$, the algebras $[,]_s$ and $[,]_u$ are isomorphic, but not equal to the original one with the value of the parameter being equal to $0$.
\end{dfn}

Thus  $\mu_t$, as given above, is a deformation of $\mathfrak{L}$ provided the following equalities hold for each $n \in \N$:
$$
\mu_t(x,\mu_t(y,z))=~\mu_t(\mu_t(x,y),z)-\mu_t( \mu_t(x,z), y)~~\mbox{for}~ x,y,z \in \mathfrak{L}.
$$
Now expanding both sides  of the above equation and collecting coefficients of $t^n$, we see that it is equivalent to the system of equations
$$
\sum_{i+j=n}\phi_i(x,\phi_j(y,z))=~\sum_{i+j=n}\{\phi_i(\phi_j(x,y),z)-\phi_i( \phi_j(x,z), y)\}~~\mbox{for}~ x,y,z \in \mathfrak{L}.
$$
 \begin{rem}
 \begin{enumerate}
\item For $n=0$, the above equality is equivalent to the usual Leibniz identity of $\mu_0= [-,-]$.
\item  For $n=1$, it is equivalent to the condition $\delta \phi_1=0$. This shows that  $\phi_1 \in CL^2(\mathfrak{L}; \mathfrak{L})$ is a cocycle.
In general, for $n\geq 2$, $\phi_n $ is just a $2$-cochain. 
\end{enumerate}
\end{rem}
 \begin{dfn}
 \begin{itemize}
\item The $2$-cochain $\phi_1$ is called the \emph{infinitesimal part} of the deformation $\mu_t$. If the Leibniz identity expression for $\mu_t$ is valid up to order $1$, which means it is satisfied up to the $t$- term, we call the deformation $\mu_t$ an \emph{infinitesimal} deformation.
\item  If no higher order terms appear, then the infinitesimal deformation gives rise to a \emph{real deformation}. 
\item A deformation $\mu_t$ is \emph{of order $n$}, if the Leibniz identity expression for $\mu_t$ holds up to order $n$, but not for higher order terms.
\end{itemize}
\end{dfn}
 \begin{prop}
The 2-cochain $\phi_1$ in the deformation $\mu_t$ is a $2$-cocycle in $HL^2 (\mathfrak{L}; \mathfrak{L}) $. 
\end{prop}
\begin{dfn}
Let $\mu_t$ and $\tilde{\mu}_t$ be two formal 1-parameter deformations of $\mathfrak{L}$. A formal isomorphism between the deformations $\mu_t$ and $\tilde{\mu}_t$ of a  Leibniz algebra $\mathfrak{L}$ is a power series $\Phi_t=\sum_{i\geq 0} \phi_i t^i$, where each $\phi_i: \mathfrak{L}\longrightarrow \mathfrak{L}$ is a $\mathbb{C}$-linear map with $\phi_0= id_{\mathfrak{L}}$ (the identity map on $\mathfrak{L}$), such that $$\tilde{\mu}_t(x,y)=~\Phi_t \circ \mu_t({{\Phi_t}^{-1}}(x),~{{\Phi_t}^{-1}}(y))~~\mbox{for}~ x,y \in \mathfrak{L}.$$

Then we say that $\mu_t$ and $\tilde{\mu}_t$ are \emph{equivalent}.
\end{dfn}
\begin{rem}
  \begin{enumerate}
\item Any deformation of $\mathfrak{L}$ that is equivalent to the original Leibniz structure $\mu_0$ is said to be a trivial  deformation.
\item The cohomology class of an infinitesimal  deformation $\mu_t$ of $\mathfrak{L}$ is determined by the equivalence class of the infinitesimal part $\phi_1$ of the deformation.
\item A Leibniz algebra $\mathfrak{L}$  is said to be \emph{rigid} if and only if every deformation of $\mathfrak{L}$  is trivial.
\item If the second cohomology space $HL^2( \mathfrak{L}; \mathfrak{L})=0$, then the Leibniz algebra $\mathfrak{L} $ is formally rigid.
 \end{enumerate}
\end{rem}
 \section{Metric Leibniz algebras in dimension $2$ and $3$}
\medskip
We do not consider the trivial case of commutative Leibniz algebras.
\smallskip

{\bf{1.}} In {\bf{dimension $2$}}, there are three non-isomorphic Leibniz algebras (see \cite{Ben, BH1}):
\begin{itemize}
\item $r_1: [e_1,e_2]=e_2, [e_2,e_1]=-e_2$;
\item $\mu_1:  [e_1,e_1]=e_2$;
\item $\mu_2:  [e_1,e_1]=e_2, [e_2,e_1]=e_2$.
\end{itemize}

Among those, only $\mu_1$ is a metric Leibniz algebra.  An invariant inner product can be given by the matrix
$$
B_{\mu_1}=\begin{pmatrix} 0&1\\
1&0
\end{pmatrix}.
$$

\begin{prop}
The Leibniz algebra $\mu_1$ has 2 non-isomorphic Leibniz 2-cocycles, and none of them define a metric Leibniz algebra with infinitesimal part being the given cocycle, nor does their linear combination.
\end{prop}
\begin{proof} 
The two Leibniz $2$-cocycles  which define Leibniz structures are:
\begin{itemize}
\item $\phi_1(e_1,e_2)=e_2$
\item  $\phi_2(e_2,e_1)=e_2$
\end{itemize}
 They both give deformations to the Leibniz algebra $\mu_2$ which is not metric.
\end{proof}
\medskip

{\bf{2.}} In {\bf{dimension $3$}}, we have one metric Lie algebra, the simple Lie algebra
$$
 sl(2,\C): [e_1,e_2]=e_3, [e_2,e_3]=2e_2, [e_1,e_3]=-2e_1
 $$
 equipped with an invariant inner product
 $$
 B_{sl(2)}=\begin{pmatrix}0&1&0\\
 1&0&0\\
 0&0&2
 \end{pmatrix},
 $$
 and $9$ non-isomorphic non-Lie complex Leibniz algebras, $3$ of them being $1$-parameter families (see \cite{RR, CK}).
Among those the simple Lie algebra $sl(2,\C)$ is rigid, so it has no Leibniz deformations either (see \cite{LP}).

From the Leibniz ones the only metric Leibniz algebra is the nilpotent Leibniz algebra
$$
\lambda_2: [e_1,e_1]=e_3
$$
which is the direct sum $\mu_1\oplus\C$. A possible invariant inner product is
$$
B_{\lambda_2}=\begin{pmatrix}0&0&1\\0&1&0\\
1&0&0
\end{pmatrix}.
$$
\begin{prop}
The Leibniz algebra $\lambda_2$ has no metric Leibniz cocycles, and it does not deform to metric Leibniz algebra.
\end{prop}
\begin{proof}
This Leibniz algebra has 8 non-equivalent Leibniz 2-cocycles, and none of them, or their linear combination, gives a metric Leibniz algebra. We list the 2-cocycles:
\begin{itemize}
\item $\phi_1(e_1,e_2)=e_1,  \phi_1(e_2,e_1)=e_2,  \phi_1(e_3,e_2)=2e_3$;
\item $\phi_2(e_1,e_2)=e_2$;
\item $\phi_3(e_2,e_1)= e_1-e_2$;
\item $\phi_4(e_2,e_1)=e_3$;
\item $\phi_5(e_2,e_2)=e_2$;
\item $\phi_6(e_2,e_2)=e_3$;
\item $\phi_7(e_3,e_1)=e_2$;
\item $\phi_8(e_3,e_1)=e_3$.
\end{itemize}
In checking the metric property for the deformed Leibniz bracket, it is often easier to check that the obtained Leibniz algebra is not symmetric. For example, $\phi_1$ and $\phi_2$ do not define right Leibniz algebras, $\phi_3$ does not give a left Leibniz algebra, while $\phi_5$, although being a metric cocycle, does not give a Leibniz algebra at all.
\end{proof}

\begin{prop}\label{only at the inf level}
For a Leibniz algebra, the 2-cocycle $\phi$ of the form $\phi(e_i, e_i)=e_i$ never defines a metric Leibniz algebra with formal $1$-parameter local base, only at the infinitesimal level.
\begin{proof}
Consider the Leibniz bracket $[-,-]_t$ defined by the cocycle $\phi$. It has to satisfy the Leibniz identity
$$
[e_i,[e_i,e_i]_t]_t = [[e_i,e_i]_t,e_i]_t - [[e_i,e_i]_t,e_i]_t 
$$
which is obviously not satisfied, as on the left hand side there is a $t^2e_i$ term, which is only zero for $t=0$. If the infinitesimal deformation is metric, with some change of the 2-cochain we might be able to eliminate the $t^2$ term, but this higher order deformation will not be metric anymore.
\end{proof}
\end{prop}
\begin{rem}
From now on, we will leave those Leibniz cocycles out.
\end{rem}

\section {Metric Leibniz algebras in dimension $4$}
The classification of 4-dimensional metric Lie algebras is known (see \cite{Ben, FP}). They are
\begin{itemize}
\item The Lie algebra $\sl(2,\C)\oplus \C $;
\item The diamond Lie algebra: $
\mathfrak{d}:[e_1,e_2]=e_3, [e_1,e_3]=-e_2, [e_2,e_3]=e_4$.
\end{itemize}
An invariant inner product for $\mathfrak{d}$ is given by
 $$
B_{\mathfrak{d}}=\begin{pmatrix}0&0&0&1\\
0&1&0&0\\
0&0&1&0\\
1&0&0&0
\end{pmatrix}.
$$ 
 Non-Lie Leibniz algebras in dimension $4$ are classified in many papers, see e.g. \cite{AOR} for nilpotent ones, and for solvable but not nilpotent ones see \cite{CK}. We use the notations in \cite{CK} and \cite{AOR}. Among the $18$ nilpotent single objects and $3$  one-parameter nilpotent families there is only one metric Leibniz algebra:
 $$
 \mathfrak{R}_{20}(0): [e_1,e_2]=e_4,\  [e_2,e_1]=e_4, [e_2,e_2]=e_3
 $$
 (this is the $0$ member of a 1-parameter family). It has
 a simple invariant inner product
 $$
 B_{\mathfrak{R}_{20}(0)}=\begin{pmatrix}0&0&1&0\\
 0&0&0&1\\
 1&0&0&0\\
 0&1&0&0
 \end{pmatrix}.
 $$
 Naturally, solvable, but not nilpotent ones are much more, but there is only one metric object among them:
 $$
 \begin{aligned}
 \mathfrak{L}_2: [e_1,e_2]=&e_3=-[e_2,e_1],\\
  [e_1,e_4]=&e_1=-[e_4,e_1],\\
   [e_4,e_2]=&e_2=-[e_2,e_4],\\
    [e_4,e_4]=&e_3
   \end{aligned}
 $$
 with a possible invariant inner product 
  $$
 B_{\mathfrak{L}_2}=\begin{pmatrix} 0&1&0&0\\
 1&0&0&0\\
 0&0&0&-1\\
 0&0&-1&0
 \end{pmatrix}.
 $$
 These are the non-decomposable cases. We also have the decomposable ones $\lambda_2\oplus\C$, $\mu_1 \oplus\mu_1$ and  $\sl(2)\oplus\C$.
 
  \begin{prop}
 The Lie algebra $\sl(2)\oplus\C$ has no metric Leibniz deformation.
 \end{prop}
 \begin{proof}
 The $2$-dimensional Leibniz cohomology space has dimension $1$ with a representative Leibniz $2$-cocycle $$\phi(e_4,e_4)= e_4.$$
 Here, no metric structure shows up with this Leibniz $2$-cocycle .

 \end{proof}
 \begin{thm}
 The diamond Lie algebra $\mathfrak{d}$ has Leibniz  deformations to $\sl(2)\oplus\C$ and to $\mathfrak{L}_2$.
 \end{thm}
  \emph{Proof.}
 The diamond Lie algebra has $4$ non-equivalent Leibniz $2$-cocycles. Two of them provide deformations to a metric Leibniz algebra. The first such Leibniz cocycle has the form
 $$
 \phi_1(e_1,e_1)=e_4,
 $$
 and the infinitesimal deformation  $\mathfrak{L_t}^1=[-,-] + t\phi_1$ is isomorphic to the Leibniz algebra $\mathfrak{L}_2$, with the change in basis elements given by
 $$
 e_1'= \sqrt{\frac{t}{2i}}( e_3+ie_2),\quad
 e_2'= \sqrt{\frac{t}{2i}}( e_3 - ie_2),\quad 
 e_3'=-te_4,\quad 
 e_4'=ie_1.
 $$
 The second Leibniz cocycle, producing metric Leibniz algebra is
$$
\phi_2(e_2,e_3)=e_1, \phi_2(e_3,e_2)=-e_1.
$$
The resulting infinitesimal deformation $\mathfrak{L}_t^2= [-,-] + t\phi_2$
is isomorphic to $\sl(2)\oplus\C$ with the basis change
$$
e_1'= \frac{1}{\sqrt{t}} (i e_2 - e_3)\quad
e_2'= \frac{1}{\sqrt{t}} (i e_2 + e_3),\quad
e_3'= 2i (e_1+ \frac{1}{t}e_4) \quad
e_4'=e_4.
$$
We found that no linear combination of $2$-cocycles gives a metric deformation.
  \begin{thm}
 The Leibniz algebra $\mathfrak{R}_{20}(0)$ has one metric deformation to $\mu_1 \oplus \mu_1$.
 \end{thm}
 \begin{proof}
 
 The second cohomology space of $\mathfrak{R}_{20}(0)$  is of dimension $9$, but only one cocycle gives metric Leibniz algebra:
 $$
 \phi(e_1,e_1)=e_3.
 $$
 The infinitesimal deformation $\mathfrak{L}_t=[-,-] + t\phi$ is isomorphic to $\mu_1 \oplus \mu_1$. The transformation can be obtained by a change of basis given by the following.
 $$
 \begin{aligned}
 e_1'=&e_1+ \sqrt{t}e_2, ~
 e_2'=& 2\sqrt{t} ( \sqrt{t} e_3+e_4)\\
 e_3'=&e_1- \sqrt{t}e_2,~
 e_4'=& 2\sqrt{t} ( \sqrt{t} e_3-e_4).
 \end{aligned}
$$
Here again, no other linear combination of $2$-cocycles gives metric deformation.
\end{proof}
 \begin{thm}
 The Leibniz algebra $\mathfrak{L}_2$ has no metric deformations.
 \end{thm}
 \begin{proof}
 The $2$-dimensional Leibniz cohomology space has dimension $2$, but no metric structure shows up with those cocycles.
 \end{proof}
 \begin{thm}
 The Leibniz algebra $\lambda_2 \oplus\C=\mu_1 \oplus\C^2$ ($[e_1,e_1]=e_3$) deforms to the metric Leibniz algebras $\mu_1 \oplus \mu_1$, to $\mathfrak{L}_2$ and $\mathfrak{R}_{20}(0)$.
 \end{thm}
 \begin{proof}
 The 2-dimensional cohomology space has dimension $27$. Out of those non-equivalent ones there are $4$ cocycles providing metric structures:
  \begin{itemize}
 \item $\phi_1(e_2,e_2)= e_2$,
 \item $\phi_2(e_2,e_2)= e_4$,
 \item $\phi_3(e_4,e_4)= e_2$,
 \item $\phi_4(e_4,e_4)= e_4$.
 \end{itemize}
 As we agreed, we leave out $\phi_1$ and $\phi_4$ (see Remark  8). The cocycles $\phi_2$ $\phi_3$ define the same Leibniz algebra with the basis change $e_2$ and $e_4$. So we only have to check one of them. 
 
 The  Leibniz algebra obtained from $\phi_2$  and $\phi_3$ is isomorphic to $\mu_1 \oplus \mu_1$ and the isomorphism can be given by the transformation
with the basis change

  $$
  e_1'= e_1, \quad e_2'= e_2,\quad e_3'= e_3,\quad e_4' = t e_4.
  $$
  The isomorphism from $\phi_3$ with $\mu_1 \oplus \mu_1$ , can be obtained by the change of basis
  $$
  e_1'=e_1, \quad e_2'= t e_2, \quad e_3'= e_3,\quad e_4'= e_4.
  $$
  As linear combination of cocycles, and basis change $e_1 \to e_4$, we get the deformation to $\mathfrak{L}_2$. With basis change $e_1 \to e_2$, we get the deformation to $\mathfrak{R}_{20}(0)$.
  
 \end{proof}
  \begin{thm}
 The Leibniz algebra $\mu_1 \oplus \mu_1$ has no metric deformations.
 \end{thm}
 \begin{proof}
 The $2$-dimensional Leibniz cohomology space has dimension $8$, but no metric structure shows up with those cocycles.
 \end{proof}
 \[
\xymatrixrowsep{0.25in}
\xymatrixcolsep{0.15in}
\xymatrix{
& \mathfrak{d}\ar[rd]\ar[ld] &  & \lambda_2\oplus \C \ar[rd]\ar[ld]& \\
sl(2,\C)\oplus \C& & \mathfrak{L}_2 &  & \mathfrak{R}_{20}(0)\ar[d]\\
 & & & &\mu_1\oplus\mu_1
 }
\]

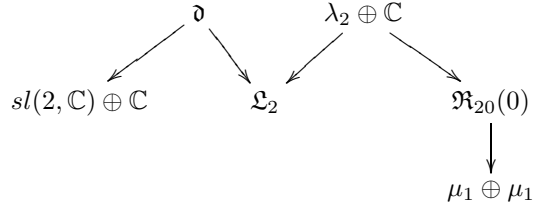
 {\bf{Figure 1:}} Metric deformations of 4-dimensional metric Leibniz algebras
 
 \. (The arrows show jump deformations.)
 
 \section{$5$-dimensional metric Leibniz algebras}
In dimension $5$, metric Lie algebras are classified in \cite{Ben, O1}. First of all, there are the decomposable  Lie algebras $ \sl(2,\C)\oplus\C^2$ and $\mathfrak{d}\oplus\C$. 
 
The third one is the indecomposable nilpotent Lie algebra
  $$
  W_3: [e_3,e_4]=e_2, \ [e_3,e_5]=e_1, \ [e_4,e_5]=e_3
  $$
  with a simple invariant inner product
  $$
  B_{W_3}=\begin{pmatrix}0&0&0&-1&0\\
  0&0&0&0&1\\
  0&0&1&0&0\\
  -1&0&0&1&0\\
  0&1&0&0&1
  \end{pmatrix}.
  $$
  We have the decomposable metric non-Lie Leibniz algebras: 
  \begin{itemize}
  \item $\lambda_2\oplus\C^2 =\mu_1\oplus\C^3$,
  \item  $\mathfrak{R}_{20}(0)\oplus\C$,
  \item $\mathfrak{L}_2\oplus\C$,
  \item $\mu_1\oplus\mu_1\oplus\C$.
  \end{itemize}
  The classification of $5$-dimensional indecomposable nilpotent Leibniz algebras is in \cite{D}. Among those there are no metric ones. Special solvable $5$-dimensional Leibniz algebras are classified in different papers, from which we could get a complete classification. Here again there are no  new indecomposable metric Leibniz algebras. With our deformation method, we found two nilpotent Leibniz algebras which were not listed in the literature.

  Let us consider the above mentioned examples case by case. The first one is obvious as in dimension four.  
  \begin{prop}
  The Lie algebra $sl(2,\C)\oplus\C^2$ has no metric Leibniz deformations.  
  \end{prop}
  
  \begin{proof} 
  The 2-dimensional Leibniz cohomology space has $8$ non-equivalent cocycles in a basis, but no metric structure shows up with those cocycles. This set comprises of the cocycle which already appeared in the four dimensional computation for $sl(2,\C)$ and there are $7$ additional cocycles as follows.
  
\begin{enumerate}
\item $\phi_1(e_4,e_4)= e_4$;
\item $\phi_2(e_4,e_4)= e_5$;
\item $\phi_3(e_4,e_5)= e_4$;
\item $\phi_4(e_4,e_5)= e_5$;
\item $\phi_5(e_5,e_4)= e_4$;
\item $\phi_6(e_5,e_4)= e_5$;
\item $\phi_7(e_5,e_5)= e_4$;
\item $\phi_8(e_5,e_5)= e_5$.
\end{enumerate}
  \end{proof} 
  
 \begin{thm}
The Lie algebra $\mathfrak{d}\oplus\C$ has Leibniz  deformations to $\sl(2)\oplus\C^2$ and $\mathfrak{L}_2\oplus \C$. Moreover, $\mathfrak{d}\oplus\C$  has one nontrivial metric infinitesimal defomation, which can not be extended to real metric deformation.
 \end{thm}
  \begin{proof} 
 The Lie algebra $ \mathfrak{d}\oplus\C$ has $13$ non-equivalent Leibniz $2$-cocycles. This set consists of $4$ skewsymmetric cocycles which already appeared in the four dimensional computation for diamond algebra. This will provide  Leibniz  deformations to $\sl(2)\oplus\C^2$ and $\mathfrak{L}_2\oplus C$. 
 
 Among the $9$ non-skewsymmetric cocycles,  only one defines an algebra with invariant inner product. 
 It is in fact a family, defined by the same nontrivial brackets as the diamond algebra, plus an additional nonzero bracket
 $$[e_5,e_5]=te_5,$$
 where $t \neq 0 $ and $e_5$ denotes the additional basis vector to the basis of diamond Lie algebra.
 
  
This family is metric, but does not define a Leibniz algebra, only at the infinitesimal level. In the Leibniz identity expression higher order terms show up:
$$
[e_5,[e_5,e_5]_t]_t - [[e_5,e_5]_t,e_5]_t - [e_5,[e_5,e_5]_t]_t = -t^2e_5,
$$
so the Leibniz identity is not satisfied already at the second level (it has an obstruction of order 2 because the $t^2$ term does not vanish). Remark that even if we change the 2-cochain to eliminate the $t^2$ term, the obtained higher order deformation will not be metric. So this does not give a real metric deformation. On the infinitesimal level, these metric algebras are isomorphic: if $t_1 \neq t_2$, with the bais change $e_5'=ce_5$ where $c^3=\frac{t_1}{t_2}$ we get an isomorphism.

 \end{proof}
 
 \rem
 Here again we have a metric cocycle which does not define a Leibniz algebra, because in the deformed algebra at the $t^2$ term does not satisfy the Leibniz identity (see Remark 8).

 \begin{thm}
 The nilpotent Lie algebra $W_3$ deforms in two different ways to the Lie algebra $\mathfrak{d}\oplus\C$, and also has a deformation to $\sl(2,\C)\oplus\C$. Moreover, it has two isomorphic metric nilpotent Leibniz algebra deformations.
 One of them is $\widetilde{W_3}$ with the nonzero brackets for $W_3$, and an extra nontrivial Leibniz bracket $[e_5,e_5]=e_2$, the other one is $\widetilde{W_3}'$ with the nonzero brackets for $W_3$ and a nonzero extra bracket $[e_4,e_4]=e_1$. Here $\widetilde{W_3}$ is isomorphic to 
 $\widetilde{W_3}'$.
 \end{thm}
 \begin{proof}
 The  second Leibniz cohomology space has $17$ non-equivalent cocycles in a basis. This set consists of $9$ skewsymmetric cocycles and $8$ non-skewsymmetric cocycles.
 The  following $3$ cocycles  give metric Lie algebras:
\begin{enumerate}
\item $ \phi_1(e_3,e_5)=e_4, \phi_1(e_5,e_3)=-e_4$;
 \item $\phi_2(e_3,e_4)=e_5,  \phi_2(e_4,e_3)=-e_5$;
 \item
$ \phi_3(e_2,e_3 )= e_2,  \phi_3(e_3,e_2)=-e_2,
 \phi_3(e_2,e_5) = -e_3,  \phi_3(e_5,e_2)= e_3, \\
 \phi_3(e_3,e_5)=e_5,  \phi_3(e_5,e_3)= -e_5.$
 \end{enumerate}
 

The deformation using the cocycle $\phi_1$ has a Leibniz algebra isomorphism with the Lie algebra $\mathfrak{d}\oplus\C$ via the transformation
  \begin{equation*}
  \begin{split}
 &e_1'=e_3+e_4+\frac{i}{\sqrt{t}}e_5, \\
 & e_2'= \frac{1}{t}e_1- \frac{ i} {\sqrt{t}}e_2+e_4, \\
 & e_3'=e_2- \frac{i}{\sqrt{t}}e_3, \\
 & e_4'=\frac{i}{\sqrt{t}}e_2~~\mbox{and}\\
 &e_5'~\mbox{is the new basis element commuting with the other basis vectors}.
  \end{split}
 \end{equation*}
 The cocycle $\phi_2$ also gives a deformation isomorphic to $\mathfrak{d}\oplus\C$  via the change of basis given below.
  \begin{equation*}
  \begin{split}
& e^{\prime}_1= e_3+ \frac{1}{\sqrt{t}}e_4 -e_5\\
&e^{\prime}_2= e_2+e_3+te_5\\
&e^{\prime}_3= (t+1)e_1 -\frac{1}{\sqrt{t}}e_2 + \sqrt{t}e_3 -\sqrt{t}e_5\\
&e^\prime_4=- \sqrt{t}(1+t)e_1~\mbox{and}\\
 &e_5'~\mbox{is the new basis element commuting with the other basis vectors}.
   \end{split}
 \end{equation*}
 The cocycle $\phi_3$ gives  gives a deformation isomorphic to $sl(2,\C)\oplus\C^2$
  via the change of basis given below.
  \begin{equation*}
  \begin{split}
& e^{\prime}_1= e_1 -e_2- \frac{2}{t}\lambda e_3 +t e_5\\
&e^{\prime}_2= \frac{-1}{2t^3}( e_1- e_2+ \frac{2}{t\lambda}e_3 +te_5)\\
&e^{\prime}_3= \lambda (e_1+e_2+ t e_5)\\
&e^\prime_4= e_4~\mbox{and}~~e_5'= e_1, ~~\mbox{where}~~\lambda =\sqrt{\frac{2}{t^3}}.
   \end{split}
 \end{equation*}
  
In the set of $8$  non-skewsymmetric cocycles, there are the following two cocycles which  give isomorphic metric Leibniz algebras.
\begin{enumerate}
\item $\phi_1(e_4,e_4)= e_1  $;
\item $\phi_2(e_5,e_5)= e_2  $.

These define the nilpotent metric Leibniz algebra $\widetilde{W_3}$.
\end{enumerate}
\end{proof}

\dfn{\bf{The new metric Leibniz algebra $\widetilde{W_3}$}}

Consider the Leibniz algebra given by a parametric family $$ \widetilde{W_3}(t) =[-,-]_{W_3} + t\phi_2$$
where $\phi_2$ denotes the $2$-cocycle $\phi_2(e_5,e_5)= e_2  $ obtained in the second Leibniz cohomology space of $W_3$. 
The algebras $ \widetilde{W_3}(t)$ are isomorphic for any $t \neq 0$, so it is just one nilpotent Leibniz algebra. Let we have two such algebras for $t_1$ and $t_2$, then an isomorphism
 $$\widetilde{W_3}(t_1) =[-,-]_{W_3} + t_1\phi_2 \longrightarrow \widetilde{W_3}(t_2) =[-,-]_{W_3} + t_2\phi_2$$
given by a basis change $e^\prime_1=c e_1, e^\prime_2= 1/c e_2, e^\prime_3= e_3, e^\prime_4=1/c e_4, e^\prime_5= c e_5$, where $c$ is a scalar such that $c^2 = t_2/t_1$.

We consider $ \widetilde{W_3} =  \widetilde{W_3}(1)$  with a simple invariant inner product
  $$
  B_{ \widetilde{W_3}}
  =\begin{pmatrix}
  0&0&0&-1&0\\
  0&0&0&0&1\\
  0&0&1&0&0\\
  -1&0&0&1&0\\
  0&1&0&0&1
  \end{pmatrix}.
  $$
Other than these, we do not get new metric Leibniz algebras with deformations.

   \begin{thm}
 The five dimensional nilpotent metric Leibniz algebra $ \widetilde{W_3}$ has two deformations to metric Leibniz algebras, with one being isomorphic to  the Leibniz algebra $\mathfrak{L}_2 \oplus\C$, the other one is the nilpotent Leibniz algebra $\widetilde{W_3}^*$ with the original nonzero brackets of $\widetilde{W_3}$ and the extra nonzero bracket $[e_4,e_4]=e_1$.
   \end{thm}
 \begin{proof}
The  second Leibniz cohomology space of $ \widetilde{W_3}$ has the following $11$ non-equivalent cocycles in a basis. 
\begin{enumerate}
\item $\psi_1:(e_1, e_3)\mapsto e_2, (e_3, e_1)\mapsto - e_2, (e_1, e_5)\mapsto -e_3; (e_5, e_1)\mapsto e_3.$
\item  $\psi_2: (e_4,e_4) \mapsto e_1$.
\item $\psi_3:$ \vspace{-0.55cm}
 \begin{equation*}
  \begin{split}
   &(e_1,e_3 ) \mapsto  e_1,   (e_3,e_1 ) \mapsto  -e_1, (e_3,e_5 ) \mapsto   e_5, (e_5,e_3 ) \mapsto  - e_5,\\
  &(e_5, e_1) \mapsto -2  e_2,  (e_1,e_4 ) \mapsto  e_3, (e_4,e_1 ) \mapsto  e_3.
    \end{split}
 \end{equation*}
  \item $\psi_4:
 (e_1,e_4 ) \mapsto  e_1,   (e_4,e_1 ) \mapsto  -e_1, (e_3,e_4 ) \mapsto   e_3, (e_4,e_3 ) \mapsto  - e_3.   $
\item $\psi_5:
 (e_1,e_4 ) \mapsto  e_2,   (e_4,e_1 ) \mapsto  -e_2, (e_4,e_5 ) \mapsto   -e_4, (e_5,e_4 ) \mapsto  e_4.   $
\item $\psi_6:$ \vspace{-0.55cm}
 \begin{equation*}
  \begin{split}
  & (e_2,e_3 ) \mapsto  e_1,   (e_3,e_2 ) \mapsto  -e_1, (e_2,e_4 ) \mapsto   e_3, (e_4,e_2 ) \mapsto  - e_3,\\
 & (e_2, e_5) \mapsto   e_2,  (e_5,e_2 ) \mapsto  -e_2, (e_4,e_5 ) \mapsto  e_4,  (e_5,e_4 ) \mapsto  -te_4\\
 &(e_5, e_5) \mapsto e_5.
    \end{split}
 \end{equation*}
 \item $\psi_7:$ \vspace{-0.55cm}
 \begin{equation*}
  \begin{split}
  & (e_2,e_3 ) \mapsto  e_2,   (e_3,e_2 ) \mapsto  -e_2,   (e_2, e_5) \mapsto -e_3,  (e_5,e_2 ) \mapsto  e_3,\\
  & (e_3,e_5 ) \mapsto  e_5,  (e_5,e_3 ) \mapsto  -e_5 , (e_5, e_5) \mapsto te_4.
    \end{split}
 \end{equation*}
 \item $\psi_8:$ \vspace{-0.55cm}
 \begin{equation*}
  \begin{split}
  & (e_3,e_3 ) \mapsto  e_1,  (e_2, e_5) \mapsto 2e_1, (e_4,e_1 ) \mapsto  -2e_1,\\
  &  (e_3,e_4 ) \mapsto  -3e_3 , (e_4, e_3) \mapsto 3e_3.
    \end{split}
 \end{equation*}
 \item $\psi_9:$\vspace{-0.55cm}
 \begin{equation*}
  \begin{split}
  & (e_3,e_3 ) \mapsto  e_2,  (e_2, e_5) \mapsto 2e_2, (e_4,e_1 ) \mapsto  -2e_2,\\
  &  (e_4,e_5 ) \mapsto  e_4 , (e_5, e_4) \mapsto -e_4.
    \end{split}
 \end{equation*}
\item $\psi_{10}: (e_4,e_4)\mapsto e_2.$
  \item $\psi_{11}: (e_4,e_5) \mapsto e_1. $
\end{enumerate}
 Here only the cocycles  $\psi_1$ and $\psi_2$ give metric Lie algebras, and no combination gives any new metric deformation.
 
The Leibniz algebra
 $$\widetilde{W_3}(t)
  = \widetilde{W_3}+ t \psi_1$$ is a deformation $ \widetilde{W_3}$ obtained  from the cocycle $\psi_1$, is isomorphic to $\mathfrak{L}_2 \oplus\C$
  via the change of basis given below.
 \begin{equation*}
  \begin{split}
& e^{\prime}_1= 
e_1+(i\sqrt{s}(s-1)+s)e_2 -i\sqrt{s}~ e_3\\
 &e^{\prime}_2= e_1+(i\sqrt{s}(s-1) - s)e_2  + i\sqrt{s} ~e_3\\
 &e^{\prime}_3=  e_2\\
 &e^{\prime}_4= e_1 +e_3+e_4+ e_5 \\  
&e^\prime_5= e_1+ s e_4.
   \end{split}
 \end{equation*}

The Leibniz algebra
 $$ \widetilde{W_3}^*(t)
 =  \widetilde{W_3} + t \psi_2$$ 
 is a deformation of $ \widetilde{W_3}$ obtained with the cocycle $\psi_2$.
\end{proof}

 \dfn{\bf{The new metric Leibniz algebra $ \widetilde{W_3}^*$}}
 
The algebras $ \widetilde{W_3}^*(t)
 =  \widetilde{W_3} + t \psi_2$ are isomorphic for any $t \neq 0$, so it is just one nilpotent Leibniz algebra. 
 

We denote $ \widetilde{W_3}^*=  \widetilde{W_3}^*(1)$  with a simple invariant inner product
  $$
  B_{ \widetilde{W_3^*}}
  =\begin{pmatrix}
  0&0&1&0&0\\
  0&0&0&1&0\\
  1&0&0&0&1\\
  0&1&0&0&0\\
  0&0&1&0&1
  \end{pmatrix}.
  $$
\begin{thm}
The Leibniz algebra $ \widetilde{W_3}^*$ has no metric deformation.
\end{thm}
\begin{proof}
The 2-dimensional Leibniz cohomology space has $5$ non-equivalent cocycles in a basis given below, but no metric structure shows up with those cocycles or with their linear combination.
\begin{enumerate}
\item $\phi_1:$  \vspace{-0.15cm}
 \begin{equation*}
  \begin{split}
  & (e_1,e_5 )  \mapsto  -2 e_2,  (e_2,e_3) \mapsto e_2, (e_3,e_2) \mapsto -e_2,  (e_2,e_4) \mapsto t e_1,  \\
&   (e_4,e_2) \mapsto -t e_1, (e_2,e_5) \mapsto -e_3,  (e_5,e_2) \mapsto e_3,  (e_3,e_4) \mapsto -2t  e_2, \\
&  (e_3,e_5) \mapsto e_5,  (e_5,e_3) \mapsto -t  e_1- e_5, (e_5,e_5) \mapsto e_4.
   \end{split}
 \end{equation*}
\item $\phi_2:$  \vspace{-0.15cm}
\begin{equation*}
  \begin{split}
  & (e_2,e_3 )  \mapsto e_1, (e_3,e_2)  \mapsto -e_1,  (e_2,e_4)  \mapsto e_3, (e_4,e_2)  \mapsto - e_3,  \\
&  (e_2,e_5)  \mapsto  e_2, (e_5,e_2)  \mapsto -e_2,  (e_3,e_4)  \mapsto -2t e_1e_3, (e_4,e_5) \mapsto e_4, \\
&  (e_5,e_4)  \mapsto - e_4,  (e_5,e_5) \mapsto e_5.
   \end{split}
 \end{equation*}
 \item $\phi_3:$ \vspace{-0.25cm}
 \begin{equation*}
  \begin{split}
  & (e_1,e_3 ) \mapsto  e_2,   (e_3,e_1 ) \mapsto  -e_2, (e_1,e_4 ) \mapsto  t e_1, (e_4,e_1 ) \mapsto  -te_1\\
 & (e_1,e_5 ) \mapsto  -e_3,   (e_5,e_1 ) \mapsto  e_3, (e_4,e_4 ) \mapsto  t e_4, (e_4,e_5 ) \mapsto  -te_5\\
 & (e_5,e_4 ) \mapsto  te_5,   (e_5,e_3 ) \mapsto  -t  e_1 +2 t  e_2.
    \end{split}
 \end{equation*}
  \item $\phi_4:$ \vspace{-0.5cm}
 \begin{equation*}
  \begin{split}
  & (e_1,e_4 ) \mapsto  -2e_1,   (e_2,e_5 ) \mapsto  2e_1, (e_3,e_3 ) \mapsto   e_1, \\
 & (e_4,e_5 ) \mapsto  3e_5, (e_5,e_4 ) \mapsto  -3e_5.
    \end{split}
 \end{equation*}
  \item $\phi_5:$ \vspace{-0.5cm}
 \begin{equation*}
  \begin{split}
  & (e_1,e_4 ) \mapsto  2e_2,   (e_2,e_5 ) \mapsto  2e_2, (e_3,e_3 ) \mapsto   e_2, \\
 & (e_3, e_4) \mapsto 2t e_1,  (e_4,e_5 ) \mapsto  -e_4, (e_5,e_4 ) \mapsto  e_4.
    \end{split}
 \end{equation*}
 \end{enumerate}
\end{proof}
\begin{thm}
The Leibniz algebra $\mathfrak{L}_2\oplus\C$ has no metric deformation.
\end{thm}
\begin{proof}
The 2-dimensional Leibniz cohomology space has $8$ non-equivalent cocycles in a basis, but no metric structure shows up with those cocycles.
\end{proof}
 \begin{thm}
 The Leibniz algebra $\mathfrak{R}_{20}(0)\oplus \C$ has one metric deformation to $\mu_1 \oplus \mu_1 \oplus \C$.
  \end{thm}
\begin{proof}
 The second cohomology space has $30$ non-equivalent cocycles. 
 This set consists of $9$ cocycles coming from cocycles already appeared in the four dimensional computation for $\mathfrak{R}_{20}(0) $. Among these only one cocycle gives metric Leibniz algebra: 
 $$
 \phi(e_1,e_1)=e_3.
 $$
 The infinitesimal deformation $\mathfrak{L}_t=[-,-] + t\phi$ is isomorphic to $\mu_1 \oplus \mu_1 \oplus \C $. 
  But no metric structure shows up with the additional $21$ cocycles or their combinations.
  \end{proof}

\begin{thm}
 The Leibniz algebra $\lambda_2\oplus \C^2$ deforms to the five dimensional Leibniz  algebra  $\mu_1\oplus\mu_1\oplus\C$ in 6 different ways, and also deforms to $\mathfrak{L}_2 \oplus \C$ and $\mathfrak{R}_{20}(0) \oplus\C$.
 \end{thm}
  \begin{proof}
  The second cohomology space has $64$ non-equivalent cocycles. This set contains $27$ cocycles coming from cocycles already appeared in the four dimensional computation for  $\lambda_2\oplus \C$. It provides two metric deformations to the five dimensional Leibniz  algebra  $\mu_1\oplus\mu_1\oplus\C$.

From the additional $37$ cocycles, there are the following five cocycles which provide metric Leibniz structures under deformation:
  \begin{itemize}
 \item $\phi_1(e_2,e_2)=e_5$;
 \item $\phi_2(e_4,e_4)=e_5$;
 \item $\phi_3(e_5,e_5)=e_5$;
 \item $\phi_4(e_5,e_5)=e_4$;
 \item $\phi_5(e_5,e_5)=e_2$
  \end{itemize}
 The Leibniz algebra obtained from them $\phi_1$ is isomorphic to 
 $\mu_1\oplus\mu_1\oplus\C$
 and an isomorphism can be given by the transformation
with the basis change
 $$
  e_1'= e_2, \quad e_2'= t e_5,\quad e_3'= -e_1,\quad e_4' = e_3, \quad e_5' = e_4.
  $$
The Leibniz algebra obtained from them $\phi_2$ is isomorphic to 
 $\mu_1\oplus\mu_1\oplus\C$
 and an isomorphism can be given by the transformation
with the basis change
  $$
  e_1'= -e_1, \quad e_2'= e_3,\quad e_3'= e_4,\quad e_4' = te_5, \quad e_5' = e_2.
  $$
 The cocycle $\phi_3$ does not define a metric Leibniz algebra (see Proposition \ref{only at the inf level}).

 The Leibniz algebra obtained from them $\phi_4$ is isomorphic to 
 $\mu_1\oplus\mu_1\oplus\C$
 and an isomorphism can be given by the transformation
with the basis change
 $$
  e_1'= -e_1, \quad e_2'= e_3,\quad e_3'= e_5,\quad e_4' = te_4, \quad e_5' = e_2.
  $$
 The Leibniz algebra obtained from them $\phi_5$ is isomorphic to 
 $\mu_1\oplus\mu_1\oplus\C$
 and an isomorphism can be given by the transformation
with the basis change
$$
  e_1'= -e_1, \quad e_2'= e_3,\quad e_3'= e_5,\quad e_4' = te_2, \quad e_5' = e_4.
  $$
  
  The jump deformations to the metric Leibniz algebras $\mathfrak{L}_2 \oplus \C$ and to $\mathfrak{R}_{20}(0) \oplus \C$ follow similarly as in the 4-dimensional case.
 \end{proof}
\begin{thm}
The Leibniz algebra $\mu_1\oplus\mu_1\oplus\C$ has no metric deformation.
\end{thm}
\begin{proof}
 The second cohomology space has $29$ non-equivalent cocycles  in a basis. 
 This set contains the $8$ cocycles coming from cocycles already appeared in the four dimensional computation for $\mu_1\oplus\mu_1 $. From the additional $21$ cocycles, there is only one metric cocycle: $\phi(e_5,e_5)=e_5$, but that does not define a Leibniz algebra (see Proposition \ref{only at the inf level}), neither do any combinations of those.
\end{proof}


\[
\xymatrixrowsep{0.25in}
\xymatrixcolsep{0.15in}
\xymatrix{
                             &                                                      &                           & W_3\ar[d]\ar[lld] &                               & \lambda_2\oplus \C^2 \ar[rd]\ar[llldd]& \\
                             & \mathfrak{d}\oplus \C \ar[rd]\ar[ld]&                           & \widetilde{W_3}\ar[rd]\ar[ld]   &                               &                              & \mathfrak{R}_{20}(0)\oplus \C \ar[d]\\
sl(2,\C)\oplus \C^2&                                                      & \mathfrak{L}_2 \oplus \C &                          & \widetilde{W_3}^*&                              &\mu_1\oplus\mu_1 \oplus \C
 }
\]


\bigskip

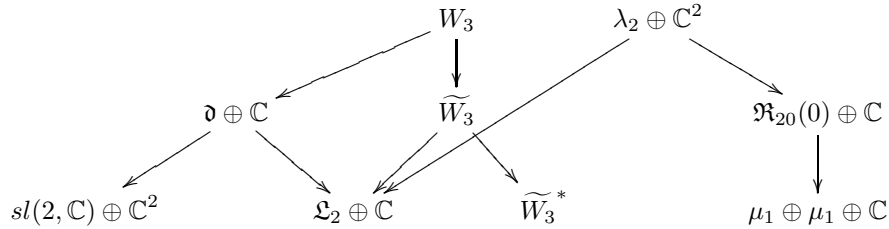
{\bf{Figure 2:}} Metric deformations of 5-dimensional metric Leibniz algebras

\bigskip

\bibliographystyle{amsplain}


\end{document}